\documentclass[12pt]{article}

\usepackage{amsmath}\usepackage{amssymb}\usepackage{amsfonts}\usepackage{amsthm}
\usepackage{color,url}
\usepackage{setspace}
\usepackage{graphicx}
\usepackage[english]{babel}\usepackage[latin1]{inputenc}\usepackage{times}\usepackage[T1]{fontenc}
\usepackage{epstopdf}
\usepackage{epsfig}
\usepackage{float}

\newtheorem{thm}{\bf{Theorem}}[section]
\newtheorem{lem}[thm]{\bf{Lemma}}

\newtheorem{cor}[thm]{\bf{Corollary}}
\newtheorem{rem}[thm]{\bf{Remark}}

\numberwithin{equation}{section}

\newcommand{\dist}{\operatorname{dist}}

\newcommand{\argmin}{\operatornamewithlimits{argmin}}

\newcommand{\R}{\operatorname{\mathbb{R}}}
\newcommand{\N}{\operatorname{\mathbb{N}}}
\newcommand{\B}{\operatorname{\mathbb{B}}}
\newcommand{\C}{\operatorname{\mathcal{B}}}
\newcommand{\randsphere}{\operatorname{{\tt randsphere}}}
\newcommand{\D}{\operatorname{\mathcal{D}}}
\newcommand{\U}{\operatorname{\mathcal{U}}}

\newcommand{\VU}{\operatorname{\mathcal{VU}}}

\newcommand{\Prox}{\operatorname{Prox}}
\newcommand{\Proj}{\operatorname{Proj}}
\newcommand{\STOPTOL}{\operatorname{s_{\tt tol}}}

\title{Computing proximal points of convex functions with inexact subgradients}

\author{W. Hare\thanks{Mathematics, University of British Columbia, Kelowna, B.C. V1V 1V7, Canada. Research by this author was partially supported by an NSERC Discovery Grant \#355571-2013. warren.hare@ubc.ca.}\and C. Planiden\thanks{Mathematics, University of British Columbia, Kelowna, B.C. V1V 1V7, Canada. Research by this author was supported by UBC UGF and by an NSERC CGS D grant. chayne.planiden@ubc.ca.}}
\date{\today}
\begin{document}

\maketitle\author
\setcounter{page}{1}\pagenumbering{arabic}

\begin{abstract}
Locating proximal points is a component of numerous minimization algorithms. This work focuses on developing a method to find the proximal point of a convex function at a point, given an inexact oracle. Our method assumes that exact function values are at hand, but exact subgradients are either not available or not useful. We use approximate subgradients to build a model of the objective function, and prove that the method converges to the true prox-point within acceptable tolerance. The subgradient $g_k$ used at each step $k$ is such that the distance from $g_k$ to the true subdifferential of the objective function at the current iteration point is bounded by some fixed $\varepsilon>0.$ The algorithm includes a novel tilt-correct step applied to the approximate subgradient.
\end{abstract}



\section{Introduction}\label{sec:intro}

Given a convex function $f : \R^n \mapsto\R,$ the proximal point of $f$ at $z$ with prox-parameter $r > 0$ is defined
    $$\Prox_f^r(z):=\argmin\limits_y\left\{f(y)+\frac{r}{2}\|y-z\|^2\right\}.$$
First arising in the works of Moreau \cite{moreau1963,proximite}, the proximal point operator has emerged as a subproblem in a diverse collection of algorithms. The basic proximal point algorithm sets
    $$x_{k+1} = \Prox_f^r(x_k)$$
and was shown to converge to a minimizer of $f$ by Martinet \cite{martreg}.  It has since been shown to provide favourable convergence properties in a number of situations (see \cite{ AKK91,  Ferris91, compprox, monops} et al.).

The basic proximal point algorithm has inspired a number of practical approaches, each of which evaluates the proximal point operator as a subroutine. For example, proximal gradient methods apply the proximal point operator to a linearization of a function \cite{Tseng10,WangXu13}.  Proximal bundle methods advance this idea by using a bundle of information to create a convex piecewise-linear approximation of the objective function, then apply the proximal point operator on the model function to determine the next iterate \cite{Bonnansetal06,redist,Kiwiel90,LemarSagas97}.  Proximal splitting methods \cite{combettespesquet11}, such as the Douglas-Rachford algorithm \cite{DouglasRachford56}, focus on the minimization of the sum of two functions and proceed by applying the proximal point operator on each function in alternation. Another example is the novel proximal method for composite functions $F(x) = g(f(x))$ \cite{LewisWright08, sagastizabal13}.  Among the most complex methods, the $\VU$-algorithm alternates between a proximal-point step and a `$\U$-Newton' step to achieve superlinear convergence in the minimization of  nonsmooth, convex functions  \cite{vualg}.

Practical implementations of all of the above methods exist and are generally accepted as highly effective. In most implementations of these methods, the basic assumption is that the algorithm has access to an oracle that returns the function value and a subgradient vector at a given point.  This provides a great deal of flexibility, and makes the algorithms suitable for nonsmooth optimization problems. However, in many applications the user has access to an oracle that returns only function values (see for example \cite{conn09,HareNutiniTesfamariam-2013} and the many references therein).  If the objective is smooth, then practitioners can apply gradient approximation techniques \cite{Kelley99}, and rely on classical smooth optimization algorithms.  However, if the objective is nonsmooth, then practitioners generally rely on direct search methods (see \cite[Ch. 7]{conn09}). While direct search methods are robust and proven to converge, they do not exploit any potential structure of the problem, so there is room for improvement and new developments of algorithms applied to nonsmooth functions using oracles that return only function values.

To that end, other papers in this vein present results that have demonstrated the ability to approximate a subgradient using only function evaluations \cite{BKS08,Gup77, HareNutini13, Kiw10}.  This provides opportunity for the development of proximal-based methods for nonsmooth, convex functions. Such methods require the use of a proximal subroutine that relies on inexact subgradient information. In this paper, we develop such a method, prove its convergence, provide stopping criterion analysis, and include numerical testing.  This method can be used as a foundation in proximal-based methods for nonsmooth, convex functions where the oracle returns an exact function value and an inexact subgradient vector.  We present the method in terms of an arbitrary approximate subgradient, noting that any of the methods from \cite{ BKS08, Gup77,HareNutini13,Kiw10}, or any future method of similar style, can provide the approximate subgradient.

\begin{rem}
It should be noted that several methods that use proximal-style subroutines and inexact subgradient vectors have already been developed \cite{deOlivetal14, Hareetal16,Kiwiel95,Shenetal15,ShenXiaPang07,SunSamCan03,hintermuller2001,solodov2007}.  However, in each case the subroutine is embedded within the developed method and only analyzed in light of the developed method.  In this paper, we develop a subroutine that uses exact function values and inexact subgradient vectors to determine the proximal point for a nonsmooth, convex function.  As a stand-alone method, the algorithm developed in this paper can be used as a subroutine in any proximal-style algorithm.  (Some particular goals in this area are outlined in Section \ref{sec:conc}.)  Some more technical differences between the algorithm in this work and the inexact gradient proximal-style subroutines in other works appear in Subsection \ref{ssec:history}.
\end{rem}

The algorithm in this paper is for finite-valued, convex objective functions, and is based on standard cutting-plane methods (see \cite{Bonnansetal06, compprox}). We assume that for a given point $x$, the exact function value $f(x)$ and an approximate subgradient $g^\varepsilon$ such that $\dist(g^\varepsilon, \partial f(x)) < \varepsilon$ are available, where $\partial f(x)$ is the convex subdifferential as defined in \cite[\S 8C]{rockwets}.  Using this information, a piecewise-linear approximation of $f$ is constructed, and a quadratic problem is used to determine the proximal point of the model function -- an {\em approximal point}.  Unlike methods that use exact subgradients, the algorithm includes a subgradient correction term that is required to ensure convergence.  The algorithm is outlined in detail in Section \ref{sec:themethod}.

The prox-parameter $r$ is fixed in this algorithm. Extension to a more dynamic prox-parameter by the method found in \cite{dynamical} should be possible; we discuss this point in the conclusion. Given a stopping tolerance $\STOPTOL$, in Section \ref{sec:converge} we prove that if a stopping condition is met, then a solution has been found that is within $\varepsilon + \STOPTOL/r$ of the true proximal point.  We also prove that any accumulation point of the algorithm lies within $\varepsilon + \STOPTOL/r$ of the true proximal point.

In Section \ref{sec:numerics}, we discuss practical implementation of the developed algorithm.  The algorithm is implemented in MATLAB version 8.4.0.150421 (R2014b) and several variants are numerically tested on a collection of randomly generated functions. Tests show that the algorithm is effective, even when $\varepsilon$ is quite large.

Finally, in Section \ref{sec:conc} we provide some concluding remarks, specifically pointing out some areas that should be examined in future research.

\section{Preliminaries}\label{sec:prelim}

\subsection{Notation}

Throughout, we assume that the objective function $f: \R^n \mapsto\R$ is finite-valued, convex, proper, and lower semicontinuous (\emph{lsc}).

The Euclidean vector norm is denoted $\|\cdot\|.$ With $\delta>0,$ we use $B_\delta(x)$ to denote the open ball centred at $x$ with radius $\delta.$ We denote the gradient of a function $f$ by $\nabla f.$ The (convex) subdifferential of $f$ is denoted $\partial f,$ and a subgradient of $f$ at $x$ is denoted $g\in\partial f(x),$ as discussed in \cite[\S 8.C]{rockwets}.  The distance from a point $x\in\R^n$ to a set $C$ is denoted $d_C(x),$ and the projection of $x$ onto $C$ is denoted $P_Cx$.

Given $K>0$, we say that the function $f$ is \emph{locally $K$-Lipschitz continuous} about $z$ with radius $\sigma > 0$, if
$$\|f(y)-f(x)\|\leq K\|y-x\|\mbox{ for all }x,y\in B_\sigma(z).$$
We say that $f$ is \emph{globally $K$-Lipschitz continuous} if $\sigma$ can be taken to be $\infty$.

\subsection{The Proximal Point}\label{sec:background}

Given a proper, lsc, convex function $f,$ a point $z\in\R^n$ (the \emph{prox-centre}), and a prox-parameter $r > 0,$ we consider the problem of finding the proximal point of $f$ at $z$:
 $$\Prox_f^r(z)=\argmin\limits_y\left\{f(y)+\frac{r}{2}\|y-z\|^2\right\}.$$
As $f$ is convex, this point exists and is unique \cite[Thm 2.26]{rockwets}.  If $f$ is locally $K$-Lipschitz continuous at $z$  with radius $\sigma > 0$, then \cite[Lemma 2]{compprox} implies
$$\|\Prox_f^r(z)-z\|<\frac{2K}{r} ~\mbox{ whenever }~ \frac{2K}{r} < \sigma.$$
The proximal point can be numerically computed via an iterative method. Given an exact oracle, one method for numerically computing a proximal point of $f$ is as follows.  Let $z\in\R^n$ be the prox-centre.  We create an information bundle, $D_k = \{ (x_i, f_i, g_i) ~:~ i \in \C_k\}$, where $x_i$ is a point at which the oracle has been called, $f_i = f(x_i)$ is the function value returned by the oracle, $g_i = g(x_i) \in \partial f(x_i)$ is the subgradient vector returned by the oracle, and $\C_k$ is the bundle index set.  At each iteration $k,$ the piecewise-linear function $\phi_k$ is defined:
$$\phi_k(x):=\max\limits_{i\in\C_k}\left\{f_i+g_i^\top(x-x_i)\right\}.$$
Then the proximal point of $\phi_k$ (the approximal point) is calculated, $x_{k+1} = \Prox_{\phi_k}^r(z)$, and the oracle is called at $x_{k+1}$ to obtain $f_{k+1}$ and $g_{k+1}$.  If $[f_{k+1}-\phi_k(x_{k+1})]/r<\STOPTOL^2$, where $\STOPTOL$ is the stopping tolerance, then the algorithm stops and returns $x_{k+1}$. Otherwise, the element $(x_{k+1},f_{k+1},g_{k+1})$ is inserted into the bundle and the process repeats.  Further information on this approach can be found in \cite[Chapter XI]{lemarechal93} and in \cite{Kiwiel90,Kiwiel95}.

Computing the approximal point is a convex quadratic program, and can therefore be solved efficiently as long as the dimension and the bundle size remain reasonable \cite{coleman1990globally}. In order to keep the bundle size reasonable, various techniques such as bundle cleaning \cite{karas2009bundle} and aggregate gradient cutting planes \cite{mayer1998stochastic} have been advanced. As a result, we have a computationally tractable algorithm that, under mild assumptions, can be proved to converge to the true proximal point.

In this work, we are interested in how this method must be adapted if, instead of returning $g_k \in \partial f(x_k)$, the oracle returns
\begin{equation}\label{gtildeclosetog}
\tilde{g}_k^\varepsilon \in \partial f(x_k)+B_\varepsilon(0).
\end{equation}
We address this issue in the next section.
\section{Replacing Exactness with Approximation}\label{sec:themethod}

\subsection{The approximate model function and approximate subgradient}

We denote the maximum subgradient error by $\varepsilon,$ and use $\tilde{g}_i^\varepsilon$ to represent the inexact subgradient returned by the oracle at point $x_i.$ We use this information to define a new bundle element to update the model function, but first we want to ensure that our model function will not lie above the objective function at the prox-centre. This is a necessary component of our convergence proof. So if the linear function defined by the new bundle element lies above $f$ at $z,$ we make a correction to $\tilde{g}_k^\varepsilon.$ We set $g_k^\varepsilon=\tilde{g}_k^\varepsilon-c_k,$ where the correction term $c_k$ is nonzero if correction is needed, zero otherwise. Then, denoting the bundle index set by $\C_k,$ the piecewise-linear model function $\phi_k^\varepsilon$ is defined:
\begin{equation}\label{phikeps}
\phi_k^\varepsilon(x):=\max\limits_{i\in\C_k}\left\{f_i+g_i^{\varepsilon\top}(x-x_i)\right\}.
\end{equation}
 We use $\D_k$ to denote the set of bundle elements. At initialization ($k=0$), we have $\C_0=\{0\}$ and $\D_0=\{(z,f_0,g_0^\varepsilon)\}.$ For each $k\geq1,$ we will have at least three bundle elements: $\C_k\supseteq\{-1,0,k\}$ and $\D_k\supseteq\{(x_k,\phi_{k-1}^\varepsilon(x_k),r(z-x_k)),(z,f_0,g_0^\varepsilon),(x_k,f_k,g_k^\varepsilon)\}.$ In bundle and cutting-plane methods, the bundle component $r(z-x_k)$ is known as the \emph{aggregate subgradient} \cite{deOlivetal14,Kiwiel95,Shenetal15}, and is an element of $\partial\phi_{k-1}^\varepsilon(x_k).$ In this work, we adopt the convention of using the index $-1$ as the label for the aggregate bundle element $(x_k,\phi_{k-1}^\varepsilon(x_k),r(z-x_k)).$We may have up to $k+2$ elements: $\C_k=\{-1,0,1,2,\ldots,k\},$ however, elements $-1,0$ and $k$ are sufficient to guarantee convergence.\\

Now let us consider the correction term $c_k.$ Suppose that
$$E_k:=f_k+\tilde{g}_k^{\varepsilon\top}(z-x_k)-f(z)>0,$$
thus necessitating a correction. We seek the minimal correction term, hence, we need to find
$$c_k\in\argmin\{\|c\|:c^\top(z-x_k)-E_k=0\}.$$
This gives
\begin{equation}\label{white}
c_k=\Proj_G(0),\mbox{ where }G=\left\{c:\frac{c^\top(z-x_k)}{\|z-x_k\|}=\frac{E_k}{\|z-x_k\|}\right\}.
\end{equation}
That is, $c_k$ is the projection of $0$ onto the hyperplane generated by the normal vector $z-x_k$ and shift constant $E_k.$ This yields
\begin{equation}\label{ck}
c_k=\frac{E_k}{\|z-x_k\|}\frac{z-x_k}{\|z-x_k\|}=E_k\frac{z-x_k}{\|z-x_k\|^2}.
\end{equation}
Now we define the approximate subgradient that we use in the algorithm:
$$g_k^\varepsilon:=\begin{cases}
\tilde{g}_k^\varepsilon-c_k,&\mbox{ if }f_k+\tilde{g}_k^{\varepsilon\top}(z-x_k)>f(z),\\
\tilde{g}_k^\varepsilon,&\mbox{ if }f_k+\tilde{g}_k^{\varepsilon\top}(z-x_k)\leq f(z).
\end{cases}$$
Since $\tilde{g}_k^\varepsilon$ is the approximate subgradient returned by the oracle but $g_k^\varepsilon$ is the one we want to use in construction of the model function, we must first prove that $g_k^\varepsilon$ also respects \eqref{gtildeclosetog}.
\begin{lem}\label{gepsbdd}
Let $f$ be convex. Then at any iteration $k,$ $\dist(g_k^\varepsilon,\partial f(x_k))<\varepsilon.$
\end{lem}
\begin{proof}  If $g_k^\varepsilon = \tilde{g}_k^\varepsilon$, then the result holds by Assumption \ref{gtildeclosetog}.  Suppose $f_k+\tilde{g}_k^{\varepsilon\top}(z-x_k)>f(z)$.  Define
\begin{align*}
H:=\left\{g:f(z)=f(x_k)+g^\top(z-x_k)\right\},\\
J:=\left\{g:f(z)\geq f(x_k)+g^\top(z-x_k)\right\}.
\end{align*}
By equation \eqref{white}, we have that $g_k^\varepsilon=P_H(\tilde{g}_k^\varepsilon).$ Since $f(z) < f_k+\tilde{g}_k^{\varepsilon\top}(z-x_k)$, we also know that $P_J(\tilde{g}_k^\varepsilon) = P_H(\tilde{g}_k^\varepsilon) = g_k^\varepsilon$.  By equation \eqref{gtildeclosetog}, there exists $\bar{g}\in\partial f(x_k)$ such that $\|\tilde{g}_k^\varepsilon-\bar{g}\|<\varepsilon.$   Since $f$ is convex, we have
$$f(z)\geq f(x_k)+\bar{g}^\top(z-x_k),$$
hence $\bar{g} \in J$ and $P_J(\bar{g}) = \bar{g}$.  Using the fact that the projection is firmly nonexpansive \cite[Proposition 4.8]{convmono}, we have
$$\|g_k^\varepsilon -  \bar{g}\| = \|P_J(\tilde{g}_k^\varepsilon)-P_J(\bar{g})\|\leq\|\tilde{g}_k^\varepsilon-\bar{g}\|<\varepsilon,$$
which is the desired result.
\end{proof}
In the case of exact subgradients, the resulting linear functions form cutting planes of $f,$ so that the model function is an underestimator of the objective function. That is, for exact subgradient $g_i,$
\begin{equation}\label{fbigger}
f(x)\geq f_i+g_i^\top(x-x_i)\mbox{ for all }i\in\C_k.
\end{equation}
In the approximate subgradient case we do not have this luxury, but all is not lost. Using inequality \eqref{fbigger} and the fact that $\|g_i^\varepsilon-g_i\|<\varepsilon,$ we have that for all $i\in\C_k$ and for all $x\in\R^n,$
\begin{align*}
f(x)&\geq f_i+(g_i-g_i^\varepsilon+g_i^\varepsilon)^\top(x-x_i)\\
&=f_i+g_i^{\varepsilon\top}(x-x_i)+(g_i-g_i^\varepsilon)^\top(x-x_i)\\
&\geq f_i+g_i^{\varepsilon\top}(x-x_i)-\|g_i-g_i^\varepsilon\|\|x-x_i\|\\
&\geq f_i+g_i^{\varepsilon\top}(x-x_i)-\varepsilon\|x-x_i\|.
\end{align*}
Hence,
\begin{equation}\label{epsfofz}
f(x)+\varepsilon\|x-x_i\|\geq\phi_k^\varepsilon(x)\mbox{ for all }x\in\R^n,\mbox{ for all }i\in\C_k,\mbox{ for all }k\in\N.
\end{equation}

\subsection{The algorithm}
Now we present the algorithm that uses approximate subgradients. In Section \ref{sec:numerics} we implement four variants of this algorithm, comparing four different ways of updating the bundle in Step 5.

\hbox{}
\textbf{\large Algorithm:\normalsize}
\begin{itemize}
\item[Step 0:]\emph{Initialization.} Given a prox-centre $z\in\R^n,$ choose a stopping tolerance $\STOPTOL\geq0$ and a prox-parameter $r>0.$ Set $k=0$
and $x_0=z.$ Set $\C_0=\{0\}.$ Use the oracle to find $f_0,\tilde{g}_0^\varepsilon.$
\item[Step 1:]\emph{Linearization.} Compute $E_k=f_k+\tilde{g}_k^{\varepsilon\top}(z-x_k)-f(z),$ and define
$$g_k^\varepsilon:=\tilde{g}_k^\varepsilon+\max\{0,E_k\}\frac{z-x_k}{\|z-x_k\|^2}.$$
\item[Step 2:]\emph{Model.} Define
$$\phi_k^\varepsilon(x):=\max\limits_{i\in\C_k}\left\{f_i+g_i^{\varepsilon\top}(x-x_i)\right\}.$$
\item[Step 3:]\emph{Proximal Point.} Calculate the point $x_{k+1}=\Prox_{\phi_k^\varepsilon}^r(z),$ and use the oracle to find $f_{k+1},\tilde{g}_{k+1}^\varepsilon.$
\item[Step 4:]\emph{Stopping Test.} If $\frac{f_{k+1}-\phi_k^\varepsilon(x_{k+1})}{r}\leq\STOPTOL^2,$ output the approximal point $x_{k+1}$ and stop.
\item[Step 5:]\emph{Update and Loop.} Create the aggregate bundle element $(x_k,\phi_{k-1}^\varepsilon,r(z-x_k)).$ Create $\C_{k+1}$ such that $\{-1,0,k\}\subseteq B_{k+1}\subseteq\{-1,0,1,2,\cdots,k\}.$ Increment $k$ and go to Step 1.
\end{itemize}

\subsection{Relation to other inexact gradient proximal-style subroutines}\label{ssec:history}

Now that the algorithm is presented, we provide some insight on how it relates to previously developed inexact subgradient proximal point computations.  First and foremost, the presented algorithm is a stand-alone method that is not presented as a subroutine of another algorithm. To our knowledge, all of the other methods of computing proximal points that use inexact subgradients are subroutines found within other algorithms.
\par In 1995, \cite{Kiwiel95} presented a method of computing a proximal point using inexact subgradients as a subroutine in a minimization algorithm for a nonsmooth, convex objective function. However, the algorithm assumes that the inexactness in the subgradient takes the form of an $\varepsilon$-subgradient. An $\varepsilon$-subgradient $v^\varepsilon$ of $f$ at $x_k$ is an approximate subgradient that satisfies
$$f(x)\geq f(x_k)+v^{\varepsilon\top}(x-x_k)-\varepsilon\mbox{ for all }x.$$
Thus, the method in \cite{Kiwiel95} relies on the approximate subgradient forming an $\varepsilon$-cutting plane in each iteration. 

While $\varepsilon$-subgradients do appear in some real-world applications \cite{fisher1985,kalvelagen2002}, in other situations the ability to determine $\varepsilon$-subgradients is an unreasonable expectation.  For example, if the objective function is a black-box that only returns function values, then subgradients could be approximated numerically using the techniques developed in \cite{BKS08,Gup77, HareNutini13, Kiw10}.  These technique will return approximate subgradients that satisfy assumption \eqref{gtildeclosetog}, but are not necessarily $\varepsilon$-subgradients.  The method of the present work changes the need for $\varepsilon$-subgradients, to  approximate subgradients that satisfy assumption \eqref{gtildeclosetog}.

Shortly before \cite{Kiwiel95} was published, a similar technique was presented in \cite{correa1993}. This version does not require the inf-compactness condition that we do, nor does it impose the existence of a minimum function value. However, it too requires the approximate subgradients to be $\varepsilon$-subgradients. A few years later \cite{ShenXiaPang07} and \cite{SunSamCan03} presented similar solvers, also as subroutines within minimization algorithms. Again, the convergence results rest upon $\varepsilon$-subgradients and model functions that are constructed using supporting hyperplanes of the objective function.
\par The algorithmic pattern in \cite{deOlivetal14} is much more general in nature; it is applicable to many types of bundle methods and oracles. In \cite{deOlivetal14}, the authors go into detail about the variety of oracles in use (upper, dumb lower, controllable lower, asymptotically exact and others), and the resulting particular bundle methods that they inspire. The oracles themselves are more generalized as well, in that they deliver approximate function values instead of exact ones. The approximate subgradient is then determined based on the approximate function value, and thus is dependent on two parameters of inexactness rather than one. The algorithm iteratively calculates proximal points as does ours, but does not include the subgradient correction step.
\par Both \cite{Shenetal15} and \cite{Hareetal16} address the issue of non-convexity. The algorithm in \cite{Shenetal15} splits the prox-parameter in two: a local convexification parameter and a new model prox-parameter. It calls an oracle that delivers an approximate function value and approximate subgradient, which are used to construct a piecewise-linear model function. That function is then shifted down to ensure that it is a cutting-planes model. In \cite{Hareetal16} the same parameter-splitting technique is employed to deal with nonconvex functions, and the oracle returns both inexact function values and inexact subgradients. The notable difference here is that the approximate subgradient is not an $\varepsilon$-subgradient; it is any vector that is within $\varepsilon$ of the subdifferential of the model function at the current point. This is the same assumption that we employ in our version. However, non-convexity forces the algorithms to involve prox-parameter corrections that obscure any proximal point subroutine. (Indeed, it is unclear if a proximal point is actually computed as a subroutine in these methods, or if the methods only use proximal directions to seek descent.)
\par In all of the above methods except for the last, the model function is built using lower-estimating cutting planes. In this work, the goal is to avoid this requirement and extend the class of inexact subgradients that can be used in these types of algorithms. The tilt-correct step in our method ensures that the model function and the objective function coincide at the prox-centre, which we show is sufficient to prove convergence in the next section. Although the last method mentioned above is for the nonconvex case and uses approximate function values, it is the most similar to the one in the present work, as it does not rely on $\varepsilon$-subgradients. The differentiating aspect  in that method, as in all of the aforementioned methods, is that it does not make any slope-improving correction to the subgradient.

\section{Convergence}\label{sec:converge}

To prove convergence of this routine, we need several lemmas that are proved in the sequel. Ultimately, we prove that the algorithm converges to the true proximal point of $f$ with a maximum error of $\frac{\varepsilon}{r}.$ Throughout this section, we denote $\Prox_f^r(z)$ by $x^*.$ To begin, we establish some properties of $\phi^\varepsilon_k.$
\begin{lem}\label{abcd}
Let $\phi_k^\varepsilon(x)=\max\limits_{i\in\C_k}\left\{f_i+g_i^{\varepsilon\top}(x-x_i)\right\}.$ Then for all $k,$
\begin{itemize}
\item[a)] $\phi_k^\varepsilon$ is a convex function,
\item[b)] $\phi_k^\varepsilon(z)=f(z),$
\item[c)] $\phi_{k+1}^\varepsilon(x)\geq\phi_k^\varepsilon(x_{k+1})+r(z-x_{k+1})^\top(x-x_{k+1})$ for all $x\in\R^n,$
\item[d)] $\phi_k^\varepsilon(x)\geq f(x_k)+g_k^{\varepsilon\top}(x-x_k)$ for all $x\in\R^n,$ and
\item[e)] $\phi_k^\varepsilon$ is $(K+\varepsilon)$-Lipschitz.
\end{itemize}
\end{lem}
\begin{proof}
\begin{itemize}
\item[a)] Since $\phi_k^\varepsilon$ is the maximum of a finite number of convex functions, $\phi_k^\varepsilon$ is convex by \cite[Proposition 8.14]{convmono}.
\item[b)] We have that $\phi_0^\varepsilon(z)=f(z)$ by definition. Then for any $k>0,$ the tilt-correct step guarantees that
$$f(x_k)+g_k^{\varepsilon\top}(z-x_k)\leq f(z),$$
so by equation \eqref{phikeps} we have $\phi_k^\varepsilon(z)\leq f(z)$ for all $k.$ Thus, we need only concern ourselves with the new linear function at iteration $k.$ For $k>0,$ either $f_k+\tilde{g}_k^{\varepsilon\top}(z-x_k)>f(z)$ or $f_k+\tilde{g}_k^{\varepsilon\top}(z-x_k)\leq f(z).$ In the former case, we make the correction to $\tilde{g}_k^\varepsilon$ so that $f_k+g_k^{\varepsilon\top}(z-x_k)=f(z).$ As for the aggregate subgradient bundle element, we have that $r(z-x_k)\in\partial\phi_k^\varepsilon(x_{k+1})$ and $\phi_k^\varepsilon$ is convex, so that $\phi_{k-1}^\varepsilon(x_k)+r(z-x_k)(x-x_k)\leq\phi_k^\varepsilon(x)$ for all $x\in\R^n.$ In particular, $\phi_{k-1}^\varepsilon(x_k)+r(z-x_k)(x-x_k)\leq\phi_k^\varepsilon(z)=f(z).$ Therefore,
$$f(z)=\phi_0^\varepsilon(z)\leq\phi_k^\varepsilon(z)=\max\limits_{i\in\C_k}\left\{f_i+g_i^{\varepsilon\top}(z-x_i)\right\}\leq f(z),$$
which proves (b).
\item[c)] Since $(x_{k+1},\phi_k^\varepsilon(x_{k+1}),r(z-x_{k+1}))\in\D_{k+1},$ we have
\begin{align*}
\phi_{k+1}^\varepsilon(x)&=\max\limits_{i\in\B_{k+1}}\left\{f_i+g_i^{\varepsilon\top}(x-x_i)\right\}\\
&\geq\phi_k^\varepsilon(x_{k+1})+r(z-x_{k+1})^\top(x-x_{k+1}).
\end{align*}
\item[d)] This is true by definition of $\phi_k^\varepsilon.$
\item[e)] We know that $f$ is locally $K$-Lipschitz, and by Lemma \ref{gepsbdd}, for each $k$ we have that $g_k^\varepsilon$ is within distance $\varepsilon$ of $\partial f(x_k).$ Therefore, $\phi_k^\varepsilon$ is (globally) $(K+\varepsilon)$-Lipschitz.
\end{itemize}
\end{proof}

\begin{rem}Lemma \ref{abcd} makes strong use of the aggregate bundle element to prove part (c).  It is possible to avoid the aggregate bundle element by setting $B_{k+1} = B_k \cup \{k\}$ at every iteration.  To see this, note that the tilt-correct step will only ever alter  $g_i^{\varepsilon}$ at iteration $i$.  As $\phi_k^\varepsilon(x) \geq \phi_k^\varepsilon(x_{k+1})+r(z-x_{k+1})^\top(x-x_{k+1})$, if $B_{k+1} = B_k \cup \{k\}$, then $\phi_{k+1}^\varepsilon(x) \geq \phi_k^\varepsilon(x)$ provides the necessary information to ensure that Lemma \ref{abcd}(c) still holds.\end{rem}
Next, we show that at every iteration the distance between the approximal point and the true proximal point is bounded by a function of the distance between $f(x_{k+1})$ and $\phi_k^\varepsilon(x_{k+1}).$ This immediately leads to an understanding of the stopping criterion.
\begin{lem}\label{lembundle2}
At every iteration of the algorithm, the distance between the proximal point of the piecewise-linear function, $x_{k+1}=\Prox_{\phi_k^\varepsilon}^r(z),$ and the proximal point of the objective function, $x^*=\Prox_f^{r}(z),$ satisfies
\begin{equation}\label{diste2r}
\dist(x^*,x_{k+1})\leq\sqrt{\frac{f(x_{k+1})-\phi_k^\varepsilon(x_{k+1})+\frac{\varepsilon^2}{4r}}{r}}+\frac{\varepsilon}{2r}.
\end{equation}
\end{lem}
\begin{proof}
Since $x_{k+1}=\Prox_{\phi_k^\varepsilon}^r(z),$ we have that
$$\phi_k^\varepsilon(x_{k+1})\leq\phi_k^\varepsilon(x)+\frac{r}{2}\|z-x\|^2\mbox{ for all }x\in\R^n.$$
By equation \eqref{epsfofz}, for $i=k+1\in\C_{k+1}$ we have
$$f(x)+\varepsilon\|x-x_{k+1}\|\geq\phi_{k+1}^\varepsilon(x)\mbox{ for all }x\in\R^n,$$
which, by Lemma \ref{abcd} (c), results in
\begin{equation}\label{tiger}
f(x)+\varepsilon\|x-x_{k+1}\|\geq\phi_k^\varepsilon(x_{k+1})+r(z-x_{k+1})^\top(x-x_{k+1})\mbox{ for all }x\in\R^n.
\end{equation}
In particular, we have
\begin{equation}\label{add1}
f(x^*)+\varepsilon\|x^*-x_{k+1}\|\geq\phi_k^\varepsilon(x_{k+1})+r(z-x_{k+1})^\top(x^*-x_{k+1}).
\end{equation}
Since $x^*=\Prox_f^{r}(z),$ we have $r(z-x^*)\in\partial f(x^*).$ Then
\begin{equation}\label{eq:xstarprox}
f(x)\geq f(x^*)+r(z-x^*)^\top(x-x^*)\mbox{ for all }x\in\R^n,
\end{equation}
thus, we have
\begin{equation}\label{add2}
f(x_{k+1})\geq f(x^*)+r(x^*-z)^\top(x^*-x_{k+1}).
\end{equation}
Adding equations \eqref{add1} and \eqref{add2} yields
\begin{align*}
f(x_{k+1})+\varepsilon\|x^*-x_{k+1}\|&\geq\phi_k^\varepsilon(x_{k+1})+r(z-x_{k+1}-z+x^*)^\top(x^*-x_{k+1}),\\
f(x_{k+1})-\phi_k^\varepsilon(x_{k+1})&\geq r\|x^*-x_{k+1}\|^2-\varepsilon\|x^*-x_{k+1}\|,\\
&=r\left[\|x^*-x_{k+1}\|^2-\frac{\varepsilon}{r}\|x^*-x_{k+1}\|+\frac{\varepsilon^2}{4r^2}-\frac{\varepsilon^2}{4r^2}\right],\\
&=r\left[\|x^*-x_{k+1}\|-\frac{\varepsilon}{2r}\right]^2-\frac{\varepsilon^2}{4r}.
\end{align*}
Isolating the squared binomial term above and taking the square root of both sides, we have
\begin{align*}
\sqrt{\frac{f(x_{k+1})-\phi_k^\varepsilon(x_{k+1})+\frac{\varepsilon^2}{4r}}{r}}&\geq\left|\|x^*-x_{k+1}\|-\frac{\varepsilon}{2r}\right|\geq\|x^*-x_{k+1}\|-\frac{\varepsilon}{2r},\\
\dist(x^*,x_{k+1})&\leq\sqrt{\frac{f(x_{k+1})-\phi_k^\varepsilon(x_{k+1})+\frac{\varepsilon^2}{4r}}{r}}+\frac{\varepsilon}{2r}.
\end{align*}
\end{proof}

\begin{rem}\label{rem:answertorefQ1}
Lemma \ref{lembundle2} not only sets up our analysis of the stopping criterion, but also provides the necessary insight to understand the algorithm's output if an early termination is invoked.  In particular, if the algorithm is used as a subroutine inside of larger method and the larger method stops the subroutine (perhaps because desirable decrease is detected), then equation \eqref{diste2r} still applies.  As such, the optimizer can still compute an error bound on the distance of the output to the true proximal point.
\end{rem}

\begin{cor}\label{cor:tol}
If the stopping criterion is satisfied, then $\dist(x^*,x_{k+1})\leq\STOPTOL+\frac{\varepsilon}{r}.$
\end{cor}
\begin{proof}
Substituting the stopping criterion into inequality \eqref{diste2r} yields
$$\dist(x^*,x_{k+1})\leq\sqrt{\STOPTOL^2+\frac{\varepsilon^2}{4r^2}}+\frac{\varepsilon}{2r}\leq\STOPTOL+\frac{\varepsilon}{r}.$$
\end{proof}
Corollary \ref{cor:tol} is our first convergence result, showing that if for some $k,$ $\phi_k^\varepsilon(x_{k+1})$ comes close enough to $f(x_{k+1})$ to trigger the stopping condition of the algorithm, then $x_{k+1}$ is within a fixed distance of the true proximal point. Now we aim to prove that the stopping condition will always be activated at some point, and the algorithm will not run indefinitely. We begin with Lemma \ref{lembundle4}, which shows that if at any iteration the new point is equal to the previous one, the stopping condition is triggered and the approximal point is within $\varepsilon/r$ of $x^*.$

\begin{lem}\label{lembundle4}
If $x_{k+2}=x_{k+1}$ for some $k,$ then the algorithm stops, and $\dist(x_{k+2},x^*)\leq\frac{\varepsilon}{r}.$
\end{lem}
\begin{proof}
We have $\phi_{k+1}^\varepsilon(x)=\max\limits_{i\in\C_{k+1}}\{f_i+g_i^{\varepsilon\top}(x-x_i)\},$ so in particular $\phi_{k+1}^\varepsilon(x_{k+1})=\max\limits_{i\in\C_{k+1}}\{f_i+g_i^{\varepsilon\top}(x_{k+1}-x_i)\}.$ Since $k+1\in\C_{k+1},$ $f(x_{k+1})+g_{k+1}^{\varepsilon\top}(x_{k+1}-x_{k+1})=f(x_{k+1}).$ Hence,
$\phi_{k+1}^\varepsilon(x_{k+1})\geq f(x_{k+1}),$ which if $x_{k+1}=x_{k+2}$ is equivalent to $\phi_{k+1}^\varepsilon(x_{k+2})\geq f(x_{k+2}),$ and the stopping criterion is satisfied. Then by Lemma \ref{lembundle2}, we have
\begin{align*}
\dist(x^*,x_{k+2})&\leq\sqrt{\frac{f(x_{k+2})-\phi_{k+1}^\varepsilon(x_{k+2})+\frac{\varepsilon^2}{4r}}{r}}+\frac{\varepsilon}{2r}\\
&\leq\sqrt{\frac{\varepsilon^2}{4r^2}}+\frac{\varepsilon}{2r}\\
&=\frac{\varepsilon}{r}.
\end{align*}
\end{proof}
Next, we prove convergence within $\varepsilon/r$ in the case that the stopping condition is never triggered and the algorithm does not stop. We show that this is true by establishing Lemmas \ref{lembundle3} through \ref{ppp}, which lead to Theorem \ref{lempconv}, the main convergence result.
\begin{lem}\label{lembundle3}
Suppose the algorithm never stops. Then the function $$\Phi(k):=\phi_k^\varepsilon(x_{k+1})+\frac{r}{2}\|z-x_{k+1}\|^2$$ is strictly increasing and bounded above.
\end{lem}
\begin{proof}
Recall that $\phi_k^\varepsilon(z)=f(z)$ by Lemma \ref{abcd} (b). Since $x_{k+1}$ is the proximal point of $\phi_k^\varepsilon$ at $z,$ we have
$$\phi_k^\varepsilon(x_{k+1})+\frac{r}{2}\|z-x_{k+1}\|^2\leq\phi_k^\varepsilon(z)+\frac{r}{2}\|z-z\|^2=f(z).$$
Therefore, $\Phi(k)$ is bounded above by $f(z)$ for all $k.$ Define
$$L_k(x):=\phi_k^\varepsilon(x_{k+1})+\frac{r}{2}\|z-x_{k+1}\|^2+\frac{r}{2}\|x-x_{k+1}\|^2.$$
Since $x_{k+1}=\Prox_{\phi_k^\varepsilon}^r(z),$ we have
\begin{align*}
L_k(x_{k+1})=\phi_k^\varepsilon(x_{k+1})+\frac{r}{2}\|z-x_{k+1}\|^2\leq\phi_k^\varepsilon(z)=f(z).
\end{align*}
By Lemma \ref{abcd} (c) with $x=x_{k+2},$ we have
\begin{equation}\label{phix2}
\phi_{k+1}^\varepsilon(x_{k+2})\geq\phi_k^\varepsilon(x_{k+1})+r(z-x_{k+1})^\top(x_{k+2}-x_{k+1}).
\end{equation}
Using inequality \eqref{phix2} we have
\begin{align*}
L_{k+1}(x_{k+2})&=\phi_{k+1}^\varepsilon(x_{k+2})+\frac{r}{2}\|z-x_{k+2}\|^2\\
&\geq\phi_k^\varepsilon(x_{k+1})+r(z-x_{k+1})^\top(x_{k+2}-x_{k+1})+\frac{r}{2}\|z-x_{k+2}\|^2\\
&=L_k(x_{k+1})+r(z-x_{k+1})^\top(x_{k+2}-x_{k+1})+\frac{r}{2}\|z-x_{k+2}\|^2-\frac{r}{2}\|z-x_{k+1}\|^2.
\end{align*}
Expanding the norms above, we have
\begin{align*}
&\frac{r}{2}\|z-x_{k+2}\|^2-\frac{r}{2}\|z-x_{k+1}\|^2\\
=&\frac{r}{2}\left[\|x_{k+2}\|^2-2x_{k+2}^\top x_{k+1}+\|x_{k+1}\|^2-2z^\top(x_{k+2}-x_{k+1})+2x_{k+2}^\top x_{k+1}-2\|x_{k+1}\|^2\right]\\
=&\frac{r}{2}\|x_{k+2}-x_{k+1}\|^2-r(z-x_{k+1})^\top(x_{k+2}-x_{k+1}).
\end{align*}
This gives us that
\begin{align*}
&L_{k+1}(x_{k+2})\\
\geq&L_k(x_{k+1})+r(z-x_{k+1})^\top(x_{k+2}-x_{k+1})+\frac{r}{2}\|x_{k+2}-x_{k+1}\|^2-r(z-x_{k+1})^\top(x_{k+2}-x_{k+1})\\
=&L_k(x_{k+1})+\frac{r}{2}\|x_{k+2}-x_{k+1}\|^2.
\end{align*}\normalsize
Since $x_{k+2}\neq x_{k+1}$ for all $k$ by Lemma \ref{lembundle4}, the equality above becomes
$$\phi_{k+1}^\varepsilon(x_{k+2})+\frac{r}{2}\|z-x_{k+2}\|^2\geq\phi_k^\varepsilon(x_{k+1})+\frac{r}{2}\|z-x_{k+1}\|^2+\frac{r}{2}\|x_{k+2}-x_{k+1}\|^2,$$
which by the definition of $\Phi$ yields
\begin{equation}\label{phik1phik}
\Phi(k+1)\geq\Phi(k)+\frac{r}{2}\|x_{k+2}-x_{k+1}\|^2>\Phi(k).
\end{equation}
Therefore, $\Phi(k)$ is a strictly increasing function.
\end{proof}
\begin{cor}\label{corbundle1}
Suppose the algorithm never stops. Then $\lim\limits_{k\rightarrow\infty}\|x_{k+1}-x_k\|=0.$
\end{cor}
\begin{proof}
By inequality \eqref{phik1phik}, we have
\begin{align*}
0&\leq\frac{r}{2}\|x_{k+2}-x_{k+1}\|^2\\
&\leq\Phi(k+1)-\Phi(k).
\end{align*}
By Lemma \ref{lembundle3}, both terms on the right-hand side above converge, and they converge to the same place. Then
$$0\leq\lim\limits_{k\rightarrow\infty}\frac{r}{2}\|x_{k+2}-x_{k+1}\|^2\leq0,$$
and since $r\neq0,$ we have that $\|x_{k+2}-x_{k+1}\|\rightarrow0.$
\end{proof}
We point out here that the sequence $\{x_k\}$ has a convergent subsequence. This is because the iterates are contained in a compact set (a ball about $z$), so that the Bolzano-Weierstrass Theorem applies. We use this fact to prove the results that follow.
\begin{lem}\label{ppp}
Let $f$ be locally $K$-Lipschitz. Suppose the algorithm never stops. Then for any accumulation point $p$ of $\{x_k\},$ $\lim\limits_{j\rightarrow\infty}\phi_{k_j}^\varepsilon(x_{{k_j}+1})=f(p),$ where $\{x_{k_j}\}$ is any subsequence converging to $p$.
\end{lem}
\begin{proof}
We have $f_{k+1}\geq\phi_k^\varepsilon(x_{k+1})$ (otherwise, the stopping criterion is satisfied and the algorithm stops). By Lemma \ref{abcd} (e), $\phi_k^\varepsilon$ is $(K+\varepsilon)$-Lipschitz:
$$\|\phi_k^\varepsilon(x_{k+1})-\phi_k^\varepsilon(x_k)\|\leq(K+\varepsilon)\|x_{k+1}-x_k\|,$$
and by Corollary \ref{corbundle1} we have that $\lim\limits_{k\rightarrow\infty}\|\phi_k^\varepsilon(x_{k+1})-\phi_k^\varepsilon(x_k)\|=0.$ We also have
\begin{align*}
f(x_{k+1})&\geq\phi_k^\varepsilon(x_{k+1})-\phi_k^\varepsilon(x_k)+\phi_k^\varepsilon(x_k)\\
&\geq\phi_k^\varepsilon(x_k)-(K+\varepsilon)\|x_{k+1}-x_k\|.
\end{align*}
By Lemma \ref{abcd} (d) with $x=x_k,$
\begin{equation}\label{eq:pk1pk}
f(x_{k+1})\geq\phi_k^\varepsilon(x_k)-(K+\varepsilon)\|x_{k+1}-x_k\|\geq f(x_k)-(K+\varepsilon)\|x_{k+1}-x_k\|.
\end{equation}
Select any subsequence $\{x_{k_j}\}$ such that $\lim\limits_{j\rightarrow\infty}x_{k_j}=p.$ Since $\lim\limits_{j\rightarrow\infty}\|x_{k_j+1}-x_{k_j}\|=0$ by Corollary \ref{corbundle1}, we have that $\lim\limits_{j\rightarrow\infty}x_{k_j+1}=p$ as well.
Hence, taking the limit of inequality \eqref{eq:pk1pk} as $j\rightarrow\infty,$ and employing Corollary \ref{corbundle1}, we have
$$f(p)\geq\lim\limits_{j\rightarrow\infty}\phi_{k_j}^\varepsilon(x_{k_j})\geq f(p).$$
Therefore, $\lim\limits_{j\rightarrow\infty}\phi_{k_j}^\varepsilon(x_{k_j})=f(p),$ and since $\lim\limits_{j\rightarrow\infty}\|\phi_{k_j}^\varepsilon(x_{{k_j}+1})-\phi_{k_j}^\varepsilon(x_{k_j})\|=0,$ we have that $$\lim\limits_{j\rightarrow\infty}\phi_{k_j}^\varepsilon(x_{{k_j}+1})=f(p).$$
\end{proof}
Now the stage is set for the following theorem, which proves that the algorithm converges to a vector that is within a fixed distance of $\Prox_f^r(z).$
\begin{thm}\label{lempconv}Let $f$ be locally $K$-Lipschitz.
Suppose the algorithm never stops. Then for any accumulation point $p$ of $\{x_k\},$ $\dist(x^*,p)\leq\frac{\varepsilon}{r}.$
\end{thm}
\begin{proof}
By inequality \eqref{tiger},
$$f(x)+\varepsilon\|x-x_{k+1}\|\geq\phi_k^\varepsilon(x_{k+1})+r(z-x_{k+1})^\top(x-x_{k+1})\mbox{ for all }x\in\R^n.$$
Select any subsequence $\{x_{k_j}\}$ such that $\lim\limits_{j\rightarrow\infty}x_{k_j}=p.$  Taking the limit as $j\rightarrow\infty$ and using Corollary \ref{corbundle1} and Lemma \ref{ppp}, we have
\begin{align}
f(x)+\varepsilon\|x-p\|&=\lim\limits_{j\rightarrow\infty}f(x)+\varepsilon\|x-x_{{k_j}+1}\|\nonumber\\
&\geq\lim\limits_{j\rightarrow\infty}\left[\phi_{k_j}^\varepsilon(x_{{k_j}+1})+r(z-x_{{k_j}+1})^\top(x-x_{{k_j}+1})\right]\nonumber\\
&=f(p)+r(z-p)^\top(x-p)\mbox{ for all }x\in\R^n.\label{fp1}
\end{align}
By equation \eqref{eq:xstarprox}, we have
\begin{equation}\label{fp2}
f(p)\geq f(x^*)+r(x^*-z)^\top(x^*-p).
\end{equation}
By equation \eqref{fp1}, in particular we have
\begin{equation}\label{fp3}
f(x^*)+\varepsilon\|x^*-p\|\geq f(p)+r(z-p)^\top(x^*-p).
\end{equation}
Adding equations \eqref{fp2} and \eqref{fp3} yields
\begin{align*}
\varepsilon\|x^*-p\|&\geq r(x^*-z+z-p)^\top(x^*-p)\\
&=r\|x^*-p\|^2.
\end{align*}
Therefore, $\dist(x^*,p)\leq\frac{\varepsilon}{r}.$
\end{proof}
Lastly, we show that the algorithm will always terminate. With the proof of Theorem \ref{stopsforsure} below, we will have proved that the algorithm does not run indefinitely, and that when it stops the output is within tolerance of the point we seek.
\begin{thm}\label{stopsforsure}
Let $f$ be locally $K$-Lipschitz. If $\|x_{k+1}-x_k\|<\frac{r\STOPTOL^2}{K+2\varepsilon},$ then the stopping condition is satisfied and the algorithm stops.
\end{thm}
\begin{proof}
By Lemma \ref{abcd} (d) with $x=x_{k+1},$ we have
\begin{align*}
\phi_k^\varepsilon(x_{k+1})&\geq f(x_k)+g_k^{\varepsilon\top}(x_{k+1}-x_k)\nonumber\\
f(x_k)-\phi_k^\varepsilon(x_{k+1})&\leq-g_k^{\varepsilon\top}(x_{k+1}-x_k)\nonumber\\
&\leq\|g_k^\varepsilon\|\|x_{k+1}-x_k\|.
\end{align*}
By Lemma \ref{abcd} (e), $\phi_k^\varepsilon$ is $(K+\varepsilon)$-Lipschitz. Hence, $\|g_k^\varepsilon\|\leq K+\varepsilon,$ and
$$f(x_{k+1})-\phi_k^\varepsilon(x_{k+1})\leq(K+\varepsilon)\|x_{k+1}-x_k\|.$$

Therefore, if $\|x_{k+1}-x_k\|<\frac{r\STOPTOL^2}{K+\varepsilon},$ then
\begin{align*}
f(x_{k+1})-\phi_k^\varepsilon(x_{k+1})&<(K+\varepsilon)\frac{r\STOPTOL^2}{K+\varepsilon}\\
\frac{f(x_{k+1})-\phi_k^\varepsilon(x_{k+1})}{r}&<\STOPTOL^2.
\end{align*}
\end{proof}
With this, we have that the sequence $\{x_k\}$ generated by the algorithm must have an accumulation point $p$ that is within $\varepsilon/r$ of $\Prox_f^r(z),$ and the algorithm will always terminate at some point $x_k$ such that $\|x_k-x^*\|\leq s_{\tt tol}+\varepsilon/r.$

\section{Numerical Tests}\label{sec:numerics}

In this section, we present the results of a number of numerical tests performed using this algorithm. The tests were run using MATLAB version 8.4.0.150421 (R2014b), on a 2.8 GHz Intel Core 2 Duo processor with a 64-bit operating system.

\subsection{Bundle variants}

We set $r=1,$ and compare four bundle variants: $3$-bundle, $(k+2)$-bundle, active bundle, and almost active bundle. In the $3$-bundle variant, each iteration uses the three bundle elements indexed by $\C_k=\{-1,0,k\}.$ In the $(k+2)$-bundle variant, we keep all the bundle elements from each previous iteration (replacing the old aggregate with the new one), and add the $k$\textsuperscript{th} element. So the bundle index set is $\C_k=\{-1,0,1,2,3,\ldots,k\},$ for a total of $k+2$ elements.\footnote{The index $-1$ is not necessary for convergence when indices $0$ through $k$ are in the bundle. However, since our convergence analysis focused on the aggregate subgradient, we keep index $-1$ in every bundle variant.} In the active bundle variant, we keep the indices $-1,0,k,$ and add in any indices $i$ that satisfy
$$\phi_k^\varepsilon(x_{k+1})=\phi_k^\varepsilon(x_i)+g_i^{\varepsilon\top}(x_{k+1}-x_i).$$
These are the linear functions that are active at iteration $k-1.$ Finally, the almost active bundle keeps the indices $-1,0,k,$ and adds in any indices that satisfy
$$\phi_k^\varepsilon(x_{k+1})<\phi_k^\varepsilon(x_i)+g_i^{\varepsilon\top}(x_{k+1}-x_i)+10^{-6}.$$
These are the linear functions that are `almost' active at iteration $k-1,$ which allows for software rounding errors to be discounted.

\subsection{Max-of-quadratics Tests}\label{moqt}

For our first tests, we use a max-of-quadratics generator to create problems. Each problem to be solved is a randomly generated function $f(x):=\max\{q_1(x),q_2(x),...,q_m(x)\},$ where $q_i$ is convex quadratic for all $i.$  There are four inputs to the generator function: $n, nf, nf_{x^*},$ and $nf_z.$ The number $n$ is the dimension of $x,$ $nf$ is the number of quadratic functions used, $nf_{x^*}$ is the number of quadratic functions that are active at the true proximal point, and $nf_z$ is the number of quadratic functions that are active at $z.$ These features can all be controlled, as seen in \cite{compprox}. The approximate subgradient is constructed by finding the gradient of the first active quadratic function, and giving it some random error of magnitude less than $\varepsilon$ by using the {\tt randsphere} routine.\footnote[1]{Randsphere is a MATLAB function that outputs a uniformly distributed set of random points in the interior of an $n$-dimensional hypersphere of radius $r$ with center at the origin. The code is found at \url{http://www.mathworks.com/matlabcentral/fileexchange/9443-random-points-in-an-n-dimensional-hypersphere/content/randsphere.m.}} That is,
$$\tilde{g}_k^\varepsilon=\nabla f_i(x)+\varepsilon\randsphere(1,n,1),$$
where $i$ is the first active index. Though we use random error, the random function is seeded, so that the results are reproducible. The primal quadratic program is solved using MATLAB's {\tt quadprog} solver.
\par Two sets of problems were generated: low-dimension trials and high-dimensions trials.  For the low-dimension trials the Hessians of the $q_i$ functions were dense and we attempted to solve ten problems at each possible state of \small$$n\in\{4, 10, 25\},~nf\in\left\{1,\left\lceil\frac{n}{3}\right\rceil,\left\lceil\frac{2n}{3}\right\rceil,n\right\},~nf_{x^*}\in\left\{1,\left\lceil\frac{n}{3}\right\rceil,\left\lceil\frac{2n}{3}\right\rceil,n\right\},~nf_z\in\left\{1,\left\lceil\frac{n}{3}\right\rceil,\left\lceil\frac{2n}{3}\right\rceil,n\right\},$$\normalsize with four variants of the algorithm and three subgradient error levels: $\varepsilon\in\{0,\STOPTOL,10\STOPTOL\}.$ This amounts to a total of 2700 problems\footnote[2]{These totals take into account that $nf_z$ and $nf_{x^*}$ cannot exceed $nf$ at any stage. When $n=4,$ we have fewer possibilities due to the low values of $\lceil\frac{n}{3}\rceil$ and $\lceil\frac{2n}{3}\rceil.$} attempted by each of the four variants, 10,800 runs altogether. For the high-dimension trials the Hessians of the $q_i$ functions were sparse, with $95\%$ density of zeros.  In high dimension, we attempted two problems at each possible state of $n\in\{100,200\}$ and the same conditions as above on the rest of the variables, for another 360 problems attempted by each variant, 1440 runs altogether.
\par The performance of the variants is presented in two performance profile graphs and a table of averages. The table provides average CPU times, average number of iterations, and average number of tilt-corrections for each of the four bundle variants. One performance profile graph is for low dimension $n=4,10,25,$ and the other is for high dimension $n=100,200.$ A performance profile is a standard method of comparing algorithms, using the best-performing algorithm for each category of test as the basis for comparison. Here, we compare the four variants based on CPU time used, and on number of iterations used, to solve each problem. We set $\STOPTOL=10^{-3}$ for all tests, and declare a problem as solved if the stopping criterion is triggered within $100n$ iterations for the low-dimension set, and $20n$ for the high-dimension set. The $x$-axis, on a natural logarithmic scale, is the factor $\tau$ of the best possible ratio, and on the $y$-axis is the percentage of problems solved within a factor of $\tau$ of the solve time of the best solve time for a given function.
\par In the low-dimension case, we see from Figure 1 that the $(k+2)$-bundle is the most efficient, and the almost active bundle follows close behind. The $3$-bundle and active bundle coincide almost exactly; their curves overlap. Comparing the results for each error level in terms of CPU time vs. number of iterations (each pair of side-by-side graphs), there is no notable difference. With the $3$-bundle and active bundle, about 22\% of the problems timed out, meaning that the upper bound of $100n$ iterations was reached before the stopping condition was triggered. With the almost active bundle, that figure drops to about 10\%, and the $(k+2)$-bundle solved all of the problems. However, it is interesting to note that about 95\% of those timed-out problems still ended with $\dist(x_k,x^*)\leq\STOPTOL+\varepsilon/r.$ While this does not contradict the theory within this paper, it suggests that the stopping test may be more difficult to satisfy than is desired. Future research should explore alternative stopping tests.\par In the high-dimension case, Figure 2 tells us that the $3$-bundle and active bundle perform well in terms of CPU time for the problems they solved, however they were only able to solve about two thirds of the problems within the allotted time limit. The $(k+2)$-bundle took more time, and the almost active bundle much more still, but the former solved all the problems and the latter solved 97\% of them. In terms of number of iterations, the same general pattern as was found in the low-dimension case appears. The only difference is that the almost active bundle uses noticeably more iterations to complete the job, but the curve still lies below that of the $(k+2)$-bundle and above the other two. The average CPU time, number of iterations, and number of tilt-corrections for both sets of problems are displayed in Table \ref{table1}. It is interesting to note that for this type of problem the tilt-correction was used sparingly in low dimension, and in high dimension it was not needed at all. Our feeling is that this is due to the way we chose to approximate subgradients; this will be made clear in the next sets of numerical tests below. Also, it needs to be mentioned that our software applied {\tt quadprog} to solve the quadratic sub-problems.  While {\tt quadprog} is readily available, it is not recognized as among the top-quality quadratic program solvers.  Testing the algorithm using different solvers might produce different results. However, it is also possible that the inexactness of the subgradients might override any improvement in solver quality.
\begin{figure}[H]\begin{center}
\includegraphics[scale=0.55]{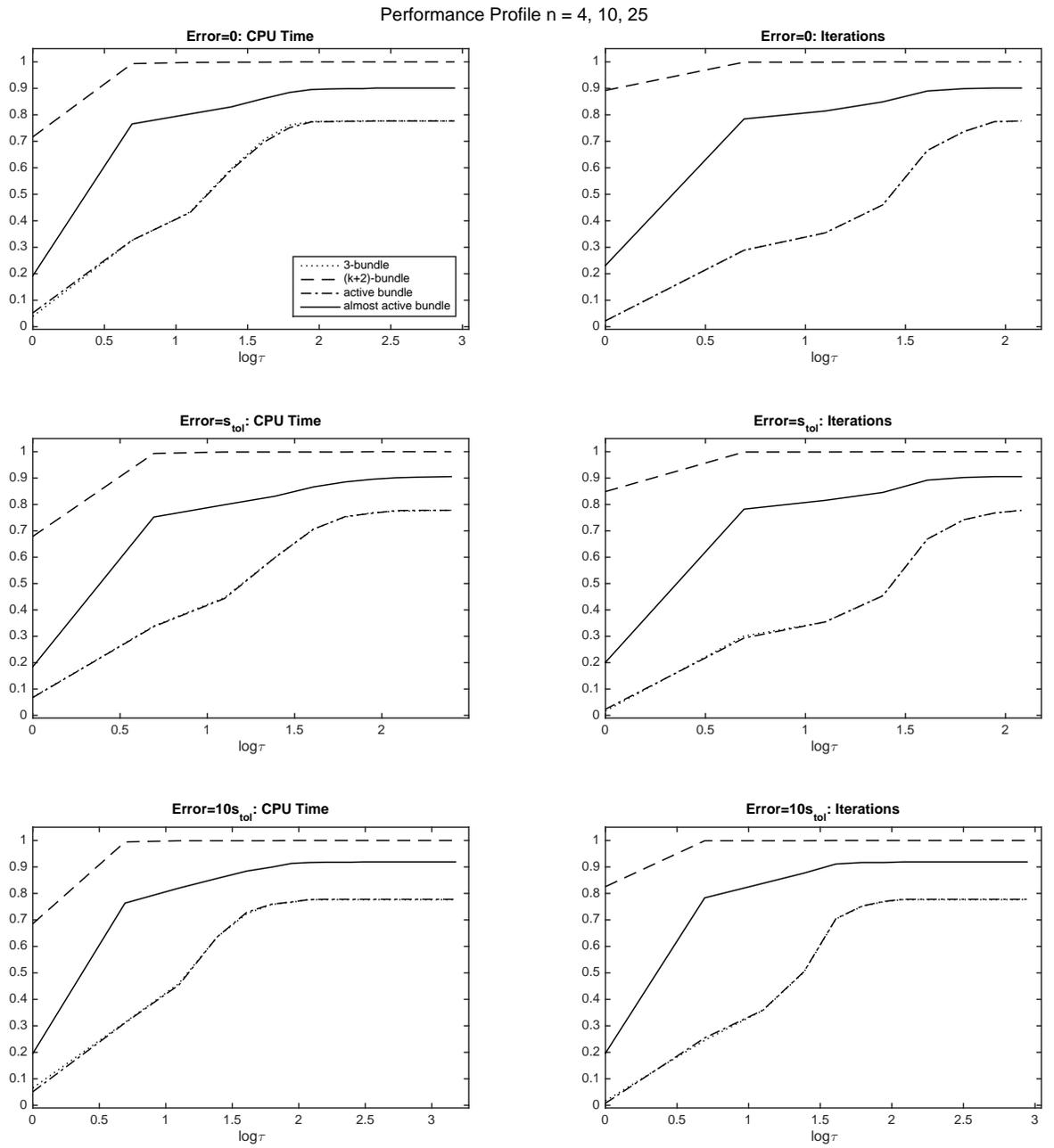}
\label{nsmalllog}
\caption{Low-dimension performance profile.}\end{center}
\end{figure}

\begin{figure}[H]\begin{center}
\includegraphics[scale=0.485]{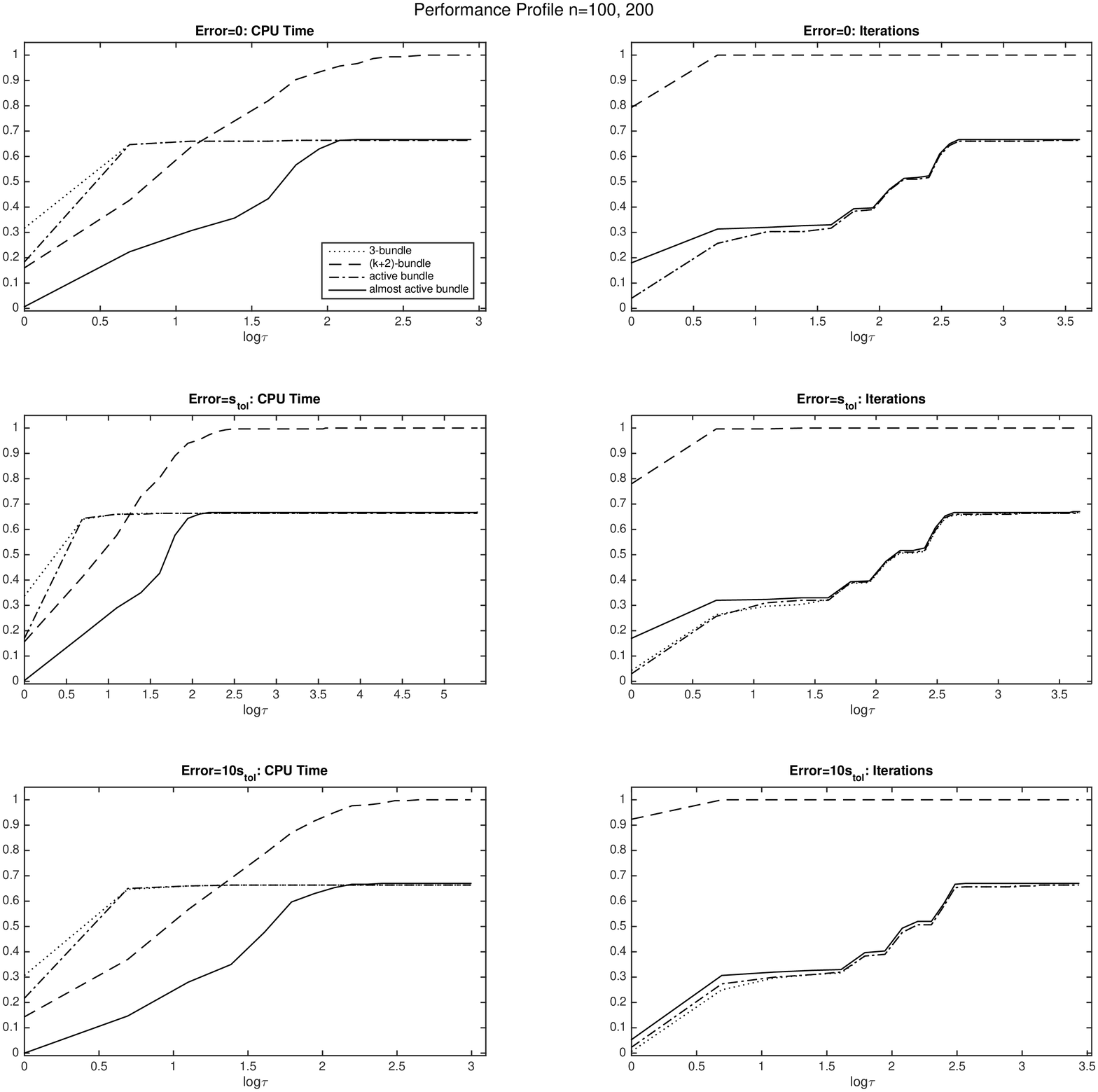}
\label{nlargelog}
\caption{High-dimension performance profile.}\end{center}
\end{figure}

\begin{table}[H]
\begin{center}
\begin{tabular}{|l|l|l|l|}
\hline
\textbf{Bundle}&\textbf{Average CPU Time}&\textbf{Average Iterations}&\textbf{Average Tilt-corrections}\\
\hline
$3$-bundle&$n$ low: 1.36s&$n$ low: 191&$n$ low: 0.0122\\
&$n$ high: 19.90s&$n$ high: 2034&$n$ high: 0\\
\hline
$(k+2)$-bundle&$n$ low: 0.61s&$n$ low: 45&$n$ low: 0.0015\\
&$n$ high: 66.31s&$n$ high: 250&$n$ high: 0\\
\hline
Active bundle&$n$ low: 1.39s&$n$ low: 191&$n$ low: 0.0111\\
&$n$ high: 19.90s&$n$ high: 2035&$n$ high: 0\\
\hline
Almost active&$n$ low: 1.96s&$n$ low: 109&$n$ low: 0.003\\
&$n$ high: 531.06s&$n$ high: 1438&$n$ high: 0\\
\hline
\end{tabular}
\end{center}
\caption{Average values among the four bundle variants.}\label{table1}
\end{table}

\subsection{Derivative-free optimization tests}

To see how the algorithm might perform in the setting of derivative-free optimization, we selected a test set of ten functions and ran the algorithm using the simplex gradient method developed in the robust approximate gradient sampling algorithm \cite{HareNutini13}. The robust approximate gradient sampling algorithm approximates a subgradient of the objective function by using convex combinations of linear interpolation approximate subgradients (details can be found in \cite{HareNutini13}). Most of the problems are taken from \cite{lukvsan2000test}, some of which were adjusted slightly to make them convex. The adjustments made and other details on these test functions appear in Appendix \ref{appendix}.   A brief description of the functions is found in Table \ref{table2}; more details are available from the authors and from \cite{lukvsan2000test}.

\begin{table}[H]
\begin{center}
\begin{tabular}{|l|l|l|l|}
\hline
\textbf{Function}&\textbf{Dimension}&\textbf{Description}&\textbf{Reference}\\
\hline
P alpha&2&max of 2 quadratics&\\
\hline
DEM&2&max of 3 quadratics&\cite[Problem 3.5]{lukvsan2000test}\\
\hline
Wong 3 (adjusted)&20&max of 18 quadratics&\cite[Problem 2.21]{lukvsan2000test}\\
\hline
CB 2&2&max of exponential, quartic, quadratic&\cite[Problem 2.1]{lukvsan2000test}\\
\hline
Mifflin 2&2&absolute value + quadratic&\cite[Problem 3.9]{lukvsan2000test}\\
\hline
EVD 52 (adjusted)&3&max of quartic, quadratics, linears&\cite[Problem 2.4]{lukvsan2000test}\\
\hline
OET 6 (adjusted)&2&max of exponentials&\cite[Problem 2.12]{lukvsan2000test}\\
\hline
MaxExp&12&max of exponentials&\\
\hline
MaxLog&30&max of negative logarithms&\\
\hline
Max10&10&max of 10\textsuperscript{th}-degree polynomials&\\
\hline
\end{tabular}
\end{center}
\caption{Set of test problems using simplex gradients.}\label{table2}
\end{table}

Each bundle variant was run 100 times on each test problem.  In all cases, the algorithm located the correct proximal point.  The average CPU time, average number of iterations, and average number of tilt-correct steps used appear in Table \ref{table2results}.

\begin{table}[H]
\begin{center}
\begin{tabular}{|l|l|l|l|}
\hline
\textbf{Bundle}&\textbf{Average CPU Time}&\textbf{Average Iterations}&\textbf{Average Tilt-corrections}\\
\hline
$3$-bundle&0.85s&147&1.01\\
\hline
$(k+2)$-bundle&1.17s&45&1.18\\
\hline
Active bundle&0.84s&146&0.93\\
\hline
Almost active&1.18s&59&0.79\\
\hline
\end{tabular}
\end{center}
\caption{Average values among the four bundle variants.}\label{table2results}
\end{table}

As in the previous test set, in terms of iterations, the $(k+2)$-bundle and the almost active variants greatly outperform the $3$-bundle and active bundle variants.  In terms of CPU time, all methods used approximately 1 second per problem.  However, unlike the first test set, this test set required an average of 1 tilt-correct step per problem. Since these latest averages were taken over varying types of functions instead of just max-of-quadratics, it could be that the tilt-correct step is less often necessary for an objective function that is not a max-of-quadratic. Or it could be that using simplex gradients, instead of finding a true subgradient and giving it some random error via the {\tt randsphere} function, requires heavier use of the tilt-correct step. This issue inspired the next set of tests below.

\subsection{Simplex gradient vs. randsphere tests}

The next set of data takes the first 2400 trials of the low-dimension case and solves the same problems using the aforementioned simplex gradient method of \cite{HareNutini13}. We compare these results with the previously obtained {\tt randsphere} method by way of the performance profile in Figure \ref{nlowsimplex} and the average values of Table \ref{tablesimp}.

\begin{figure}[H]\begin{center}
\includegraphics[scale=0.485]{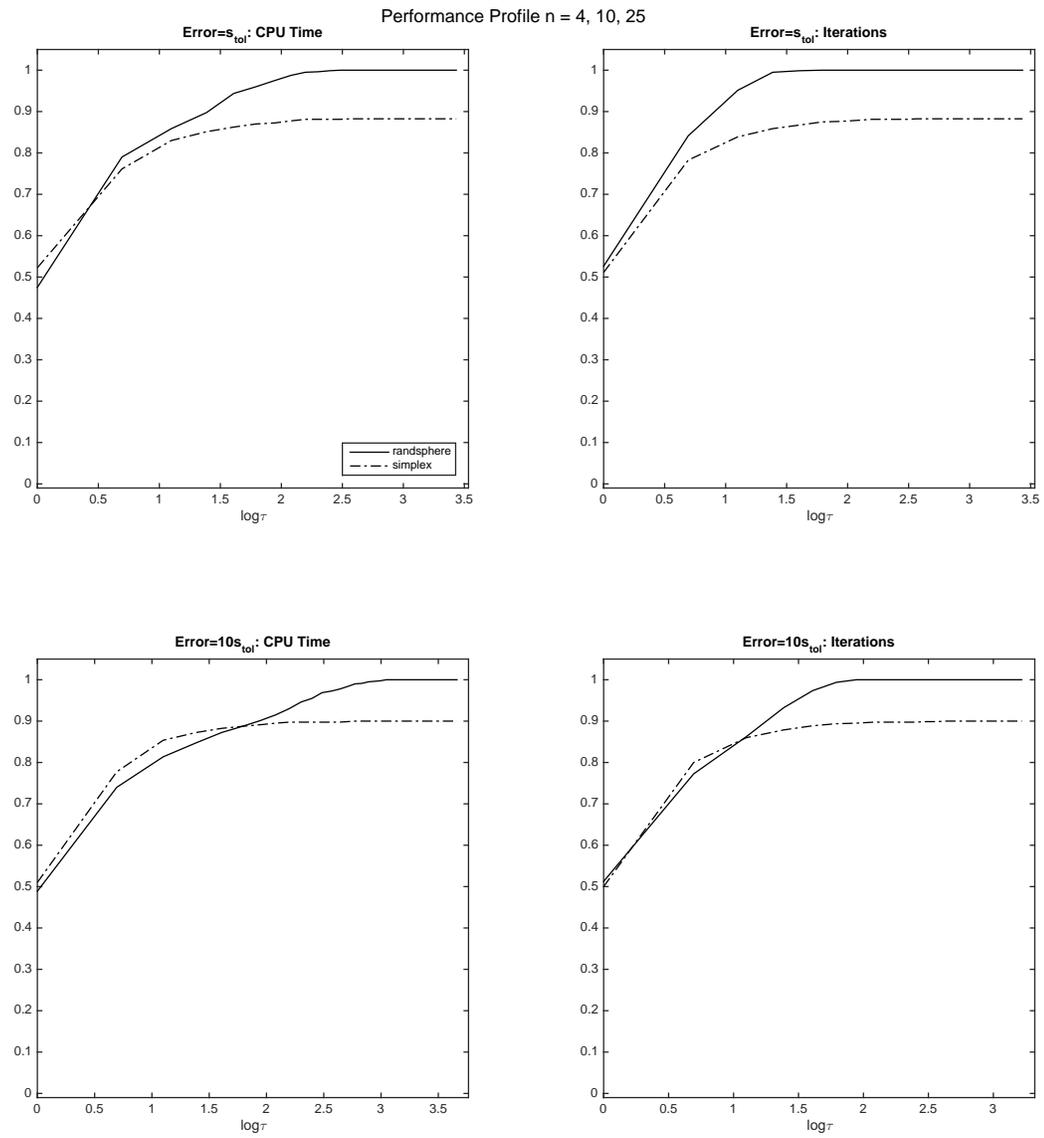}
\label{nlowsimplex}
\caption{Low-dimension performance profile -- {\tt randsphere} vs. simplex gradient.}\end{center}
\end{figure}

\begin{table}[H]
\begin{center}
\begin{tabular}{|l|l|l|l|}
\hline
\textbf{Subgradient Method}&\textbf{Average CPU Time}&\textbf{Average Iterations}&\textbf{Average Tilt-corrections}\\
\hline
{\tt randsphere}&0.6606&49.9050&0.0040\\
Simplex&0.8034&44.6333&12.2750\\
\hline
\end{tabular}
\end{center}
\caption{Average values for two methods of approximating subgradients.}\label{tablesimp}
\end{table}
 There is not a very noticeable difference in the performance profile graph, except that when using the {\tt randsphere} method all the problems are solved, whereas about 10\% of problems timed out when the simplex gradient method was used. As in Section \ref{moqt}, a problem times out if more than $100n$ iterations are required to trigger the stopping condition. The two curves start out in the same place and increase at about the same rate. The table reflects that; both the average CPU time and the average number of iterations have negligible differences between the two methods. However, we do notice a large difference in the average number of tilt-corrections used. As mentioned in the discussion of the previous data set, the tilt-correct step is almost never implemented when using {\tt randsphere} to give a true subgradient some error. However, when the simplex gradient method is used there is an average of 12 tilt-corrections performed per problem solved. This suggests that in a true DFO setting, the tilt-correct step will be utilized much more often.

\section{Conclusion}\label{sec:conc}

We have presented an approximate subgradient method for numerically finding the proximal point of a convex function at a given point. The method assumes the availability of an oracle that delivers an exact function value and a subgradient that is within distance $\varepsilon$ of the true subdifferential, but does not depend on the approximate subgradient being an $\varepsilon$-subgradient, nor on the model function being a cutting-planes function (one that bounds the objective function from below). The method is proved to converge to the true prox-point within $\STOPTOL+\varepsilon/r,$ where $\STOPTOL$ is the stopping tolerance of the algorithm, $\varepsilon$ is the bound on the approximate subgradient error, and $r$ is the prox-parameter.
\par From a theoretical standpoint, two questions immediately present themselves.  First, could the method be extended to work for nonconvex functions? Second, could the method be extended to work in situations where the function value is inexact?  Some of the techniques in this paper were inspired by \cite{compprox,Hareetal16}, which suggests that the answer to both questions could be positive.  However, the extensions are not as straightforward as they may first appear.  The key difficulty in extending this algorithm in either of these directions is that, when multiple potential sources of error are present, it is difficult to determine the best course of action.  For example, suppose $f_k+\tilde{g}_k^{\varepsilon\top}(z-x_k)-f(z)>0.$ Then, by convexity, we know the inexactness of the subgradient is at fault and we perform a tilt-correct step.  However, if the function is nonconvex, then the above inequality could occur due to nonconvexity or due to the inexactness of the subgradient.  If the error is due to inexactness, then a tilt-correct step is still the right thing to do. If, on the other hand, the error is due to nonconvexity, then it might be better to redistribute the prox-parameter as suggested in \cite{compprox}.  These issues are equally complex if inexact function values are allowed, and even more complex if both nonconvex functions and inexact function values are permitted.
\par Another obvious theoretical question would be what happens if $\varepsilon_k$ asymptotically tends to $0$?  Would the algorithm converge to the exact proximal point?  It is likely that, because past information is preserved in the aggregate subgradient, allowing $\varepsilon_k\rightarrow0$ inside this routine will not result in an asymptotically exact algorithm.  However, this is not a concern, as the purpose of this algorithm is to be used as a subroutine inside a minimization algorithm, where $\varepsilon_k$ can be made to tend to zero outside the proximal routine. This has the effect of resetting the routine with ever smaller values of $\varepsilon,$ which yields asymptotic exactness.
\par One may also wonder what the effect of changing the prox-parameter as the algorithm runs would have on the results and the speed of convergence. In \cite{dynamical}, the authors outline a method for dynamically choosing the optimal prox-parameter at each iteration by solving an additional minimization problem and incurring negligible extra cost, with encouraging numerical results. In future work, that method could be incorporated into this algorithm to see if the runtime improves.
\par From a numerical standpoint, we found that when subgradient errors are present, it is best to keep all the bundle information and use it at every iteration. The other bundle variants also solve the problems, but clearly not as robustly as the biggest bundle does.\par There are many more numerical questions that could be investigated. For example, error bounds/accuracy on the results could be analyzed, and as mentioned above, the effect of a dynamic prox-parameter could be investigated. The more immediate goal is to examine this new method in light of the $\VU$-algorithm \cite{vualg}, which alternates between a proximal step and a quasi-Newton step to minimize nonsmooth convex functions. The results of this paper theoretically confirm that the proximal step will work in the setting of inexact subgradient evaluations.  Combined with \cite{Hare2014}, which theoretically examines the quasi-Newton step, the development of a $\VU$-algorithm for inexact subgradients appears achievable.

\subsection*{Acknowledgements} 

We are grateful for the excellent feedback from two anonymous referees that helped improve this manuscript. 

\appendix

\section{Test Problem Details}\label{appendix}

Three of the test functions taken from \cite{lukvsan2000test} were adjusted to make them convex. Four other test functions did not come from \cite{lukvsan2000test}. Those details are the following. If a function does not appear in this list, it is unchanged from \cite{lukvsan2000test}.
\begin{itemize}
\item[(i)] \emph{P alpha.} This function is defined $P_\alpha:\R^2\rightarrow\R,$
$$P_\alpha(x):=\max\{x_1^2+\alpha x_2,x_1^2-\alpha x_2\}.$$
For this test, $\alpha$ was set to 10,000.
\item[(ii)] \emph{Wong 3 adjusted.} From Wong 3 \cite[Problem 2.21]{lukvsan2000test}, the term $x_1x_2$ was removed from $f_1,$ and the term $-2x_1x_2$ was removed from $f_5.$ All other sub-functions remained the same.
\item[(iii)] \emph{EVD 52 adjusted.} From EVD 52 \cite[Problem 2.4]{lukvsan2000test}, the term $2x_1^3$ in $f_5$ was changed to $2x_1^4.$ All other functions remained the same.
\item[(iv)] \emph{OET 6 adjusted.} From OET 6 \cite[Problem 2.12]{lukvsan2000test}, the term $x_1e^{x_3t_i}$ was changed to $x_1+e^{x_3t_i},$ and the term $x_2e^{x_4t_i}$ was changed to $x_2+e^{x_4t_i}$ for all $i.$
\item[(v)] \emph{MaxExp.} This function is defined $f:\R^{12}\rightarrow\R,$
$$f(x)=\max \limits_{i\in\{1,2,\ldots,12\}}\left\{ie^{x_i}\right\}.$$
\item[(vi)] \emph{MaxLog.} This function is defined $f:\R^{30}_+\rightarrow\R,$
$$f(x)=\max \limits_{i\in\{1,2,\ldots,30\}}\left\{-i\ln x_i\right\}.$$
\item[(vii)] \emph{Max10.} This function is defined $f:\R^{10}\rightarrow\R,$
$$f(x)=\max \limits_{i\in\{1,2,\ldots,10\}}\left\{ix_i^{10}\right\}.$$
\end{itemize}

\def\cprime{$'$}

\end{document}